\documentclass[a4paper,11pt]{amsart}

\usepackage{amsmath}
\usepackage{amssymb}
\usepackage{amsfonts}

\newtheorem{thm}{Theorem}[section]
\newtheorem{prop}[thm]{Proposition}
\newtheorem{lem}[thm]{Lemma}
\newtheorem{cor}[thm]{Corollary}
\newtheorem{rem}[thm]{Remark}
\newtheorem{rems}[thm]{Remarks}
\newtheorem{defi}[thm]{Definition}
\newtheorem{exo}{\bf\large Exercice}

\newcommand{\R}{\mathbb{R}}
\newcommand{\Z}{\mathbb{Z}}
\newcommand{\N}{\mathbb{N}}
\newcommand{\C}{\mathbb{C}}

\newcommand{\Sum}{\displaystyle \sum}

\newcommand{\Int}{\displaystyle \int}

\newcommand{\Inf}{\displaystyle \inf}

\newcommand{\Lim}{\displaystyle \lim}

\newcommand{\ds}{\displaystyle}

\newcommand{\beq}{\begin{eqnarray}}
\newcommand{\eeq}{\end{eqnarray}}
\newcommand{\bq}{\begin{equation}}
\newcommand{\eq}{\end{equation}}
\newcommand{\beqn}{\begin{eqnarray*}}
\newcommand{\eeqn}{\end{eqnarray*}}
\newcommand{\bex}{\begin{exo}}
\newcommand{\eex}{\end{exo}}
\newcommand{\ben}{\begin{enumerate}}
\newcommand{\een}{\end{enumerate}}

\let\wt=\widetilde
\let\wh=\widehat

\def\inte#1{
\displaystyle\mathop{#1\kern0pt}^\circ }

\def\cD{{\mathcal D}}

\def\cL{{\mathcal L}}

\def\cP{{\mathcal P}}

\def\cR{{\mathcal R}}

\date{\today}
\begin{document}
%@@@@@@@@@@@@@@@@@@@@@@@@@@@@@@@@@@@@%@@@@@@@@@@@@@@@@@@@@@@@@@@@@@@@@@@@@%@@@@@@@@
 \title[A Fourier   approach to the profile decomposition in   Orlicz spaces]{A Fourier   approach to the profile decomposition in   Orlicz spaces}

\author[H. Bahouri]{Hajer Bahouri}
\address[H. Bahouri]%
{Laboratoire d'Analyse et de Math{\'e}matiques Appliqu{\'e}es UMR 8050 \\
Universit\'e Paris-Est  Cr{\'e}teil \\
61, avenue du G{\'e}n{\'e}ral de Gaulle\\
94010 Cr{\'e}teil Cedex, France}
\email{hbahouri@math.cnrs.fr}
\author[G. Perelman]{Galina Perelman}
\address[G. Perelman]%
{Laboratoire d'Analyse et de Math{\'e}matiques Appliqu{\'e}es UMR 8050 \\
Universit\'e Paris-Est  Cr{\'e}teil\\
61, avenue du G{\'e}n{\'e}ral de Gaulle\\
94010 Cr{\'e}teil Cedex, France}
\email{galina.perelman@u-pec.fr}
 
\subjclass[2010]{}

 \maketitle
\begin{abstract}
This paper is devoted to the characterization of the lack of compactness of the Sobolev embedding of $H^N(\R^{2N})$ into the Orlicz space using  Fourier analysis. The approach adopted in this article is strikingly different from the one used in 2D   in \cite{BMM, BMM2, BMM1}, which consists  in tracking the large values of  the  sequences considered. The analysis we employ in this work is inspired by the strategy of P. G\'erard in  \cite{G} and is based on the notion  introduced in \cite{H} of being  $\log$-oscillating  with respect to a scale.

\end{abstract}

%@@@@@@@@@@@@@@@@@@@@@@@@@@@@@@@@@@@@%@@@@@@@@@@@@@@@@@@@@@@@@@@@@@@@@@@@@

%@@@@@@@@@@@@@@@@@@@@@@@@@@@@@@@@@@@@%@@@@@@@@@@@@@@@@@@@@@@@@@@@@@@@@@@@@%@@@@@@@@

%@@@@@@@@@@@@@@@@@@@@@@@@@@@@@@@@@@@@%@@@@@@@@@@@@@@@@@@@@@@@@@@@@@@@@@@@@%@@@@@@@@@@@@@@
\maketitle 

%\tableofcontents

%@@@@@@@@@@@@@@@@@@@@@@@@@@@@@@@@@@@@%@@@@@@@@@@@@@@@@@@@@@@@@@@@@@@@@@@@@%@@@@@@@@@@@@@@@

\section{Introduction and statement of   the results}
%@@@@@@@@@@@@@@@@@@@@@@@@@@@@@@@@@@@@%@@@@@@@@@@@@@@@@@@@@@@@@@@@@@@@@@@@@
\subsection{Setting of the problem and main result}
We are interested in the nonhomogeneous Sobolev space $H^{N}(\R^{2N})$ which consists of functions $u$ in $L^2(\R^{2N})$ such that  
\begin{equation} \|u\|^2_{H^N(\R^{2N})}:= \int_{\R^{2N}} (1+|\xi|^2)^N |\wh {u }(\xi)|^2 \, d \xi < \infty\, ,\label{defsob}\end{equation} where  $\wh {u }$ denotes the Fourier transform of $u$:
$$ \wh {u }(\xi):= \int_{\R^{2N}}\,  {\rm e}^{- i \, x \cdot \xi} \, u(x)\, dx\,,$$ 
with $x \cdot \xi$  the inner product on $\R^{2N}$. It is well known that  $H^{N}(\R^{2N})$ is continuously embedded  in all
Lebesgue spaces $L^p(\R^{2N})$ for $2\leq p<\infty$, but not in
$L^{\infty}(\R^{2N} )$. On the other hand, it
is also known  that  $H^N(\R^{2N})$ embeds in ${\mathcal L}(\R^{2N})$,  where ~${\mathcal L}(\R^{2N})$ denotes the Orlicz space associated to the function $\phi(s)={\rm e}^{ s^2}-1$.  Recall that   the Orlicz spaces are defined as follow: \begin{defi}\label{deforl}\quad\\
Let $\phi : \R^+\to\R^+$ be a convex increasing function such that
$$
\phi(0)=0=\lim_{s\to 0^+}\,\phi(s),\quad
\lim_{s\to\infty}\,\phi(s)=\infty\,.
$$
We say that a measurable function $u : \R^d\to\C$ belongs to
$L^\phi$ if there exists $\lambda>0$ such that
$$
\Int_{\R^d}\,\phi\left(\frac{|u(x)|}{\lambda}\right)\,dx<\infty\,.
$$
We denote then \bq \label{norm}
\|u\|_{L^\phi}=\Inf\,\left\{\,\lambda>0,\quad\Int_{\R^d}\,\phi\left(\frac{|u(x)|}{\lambda}\right)\,dx\leq
1\,\right\}\,. \eq
\end{defi}
\noindent  The Sobolev embedding
 \begin{equation}H^N(\R^{2N})\hookrightarrow {\mathcal L}(\R^{2N})\,, \label{embedorliczNN}\end{equation}
 stems immediately from the
following sharp Moser-Trudinger type inequalities (see \cite{AT,A,Ruf,sharp} for further details):
\begin{prop}\label{proptmN}
\begin{equation}
\label{Mos2} \sup_{\|u\|_{H^N(\R^{2N})}\leq 1}\;\;\int_{\R^{2N}}\,\left({\rm
e}^{\beta_N |u(x)|^2}-1\right)\,dx <\infty,
\end{equation}
where $\ds \beta_N= \frac{2N \pi^{2N} 2^{2N}}{\omega_{2N-1}}$, with  $\ds \omega_{2N-1}= \frac {2 \, \pi^{N}} {(N-1)!}$ the measure of the sphere ${\mathbb  S}^{2N-1}$. 
\end{prop} 

The aim of  this paper  is  to  describe the lack of compactness of the Sobolev embedding  of $H^N(\R^{2N})$ into ${\mathcal L}(\R^{2N})$. 
\noindent In  \cite{BMM, BMM2, BMM1},  
the authors   studied this question in the 2D particular case   
and   characterized the lack of compactness  by means   of an asymptotic decomposition in terms  of generalizations of the following
example by Moser (\cite{Lions1, Lions2,M}): 
\begin{eqnarray*}
 f_{\alpha_n}(x)&=&\; \left\{
\begin{array}{cllll}\sqrt{\frac{\alpha_n}{2\pi}}\quad&\mbox{if}&\quad |x|\leq {\rm
e}^{-\alpha_n},\\\\ -\frac{\log|x|}{\sqrt{2\alpha_n\pi}} \quad
&\mbox{if}&\quad {\rm e}^{-\alpha_n}\leq |x|\leq 1 ,\\\\
0 \quad&\mbox{if}&\quad
|x|\geq 1,
\end{array}
\right.
\end{eqnarray*}
where $(\alpha_n)$ is a sequence of positive real numbers going to infinity.   
\noindent In order to state in a clear way the corresponding  result, let us  introduce some notations as in \cite{BMM1}. 
\begin{defi}
\label{orthogen}  We shall  designate by  a scale any sequence $\underline{\alpha}:=(\alpha_n)$
of positive real numbers going to infinity, by   a core any  sequence $\underline{x}:=(x_n)$  of points in $\R^{2N}$ and by a profile any function $ {\varphi}$ belonging to $L^2(\R_+)$.
Given two scales $\underline{\alpha}$, $\underline{\tilde \alpha}$,  two cores $\underline{x}$, $\underline{\tilde x}$ and  two profiles $ {\varphi }$,  $ {\wt \varphi }$, we  shall say that the triplets  $(\underline{\alpha},\underline{x},  {\varphi })$ and $(\underline{\tilde \alpha},\underline{\tilde x},  {\wt \varphi })$, 
 are orthogonal  if
\begin{equation}\label{caseI}  \mbox{either}\qquad \Big|\log\left({\tilde \alpha_n}/{\alpha_n}\right)\Big|\stackrel{n\to\infty}\longrightarrow \infty \,,\end{equation}
or $  \tilde \alpha_n =  \alpha_n $ and
\begin{equation}\label{caseII} - \frac{ \log|x_n- \tilde x_n|}{\alpha_n} \stackrel{n\to\infty}\longrightarrow a \in [-\infty,+\infty[  \,,\end{equation}
with in the case when $\ds a  \in \,] 0,+\infty[$ $\psi$ or ${\tilde \psi}$ null for $s < a$, where $\ds \psi(s):= \int^s_0  {\varphi}(t)  \, dt$ and $\ds {\tilde \psi}(s):= \int^s_0  {\tilde \varphi}(t)  \, dt$.
\end{defi}
In \cite{ BMM1}, the following result is proved. 
\begin{thm}
\label{noradmain} Let $(u_n)$ be a bounded sequence in
$H^1(\R^2)$ such that \bq \label{noradmain-assum1}
u_n\rightharpoonup 0  \quad
\mbox{in}  \quad H^1(\R^{2}), \quad n\to\infty, \eq \bq \label{noradmain-assum2}
\limsup_{n\to\infty}\|u_n\|_{\mathcal L}=A_0 >0 \quad \quad
\mbox{and}\eq
\bq \label{noradmain-assum4} \lim_{R\to\infty}\;
\limsup_{n\to\infty}\,\|u_n\|_{\mathcal L (|x|>R)}=0. \eq Then, up to a subsequence extraction, there
exist a sequence of scales $(\underline{\alpha}^{(j)})$, a  sequence of cores $(\underline{x}^{(j)})$ and a sequence  of profiles $( {\varphi^{(j)} })$ such that the triplets $(\underline{\alpha}^{(j)}, \underline{x}^{(j)},\varphi^{(j)})$ are pairwise
orthogonal in the sense of Definition \ref{orthogen} and,  we have for all
$\ell\geq 1$, \bq \label{noraddecomp}
u_n(x)=\Sum_{j=1}^{\ell}\,\sqrt{\frac{\alpha_n^{(j)}}{2\pi}}\;\psi^{(j)}\left(\frac{-\log|x - x_n^{(j)}|}{\alpha_n^{(j)}}\right)+{\rm
r}_n^{(\ell)}(x),\quad\limsup_{n\to\infty}\;\|{\rm
r}_n^{(\ell)}\|_{\mathcal
L}\stackrel{\ell\to\infty}\longrightarrow 0, \eq
where $\ds \psi^{(j)}(s):= \int^s_0  {\varphi^{(j)}}(t)  \, dt$.  \\ 
Moreover for all
$\ell\geq 1$, we have
the following orthogonality equality  \bq \label{ortogonal2} \|\nabla
u_n\|_{L^2(\R^2)}^2=\Sum_{j=1}^{\ell}\,\|\varphi^{(j)}\|_{L^2(\R_+)}^2+\|\nabla
{\rm r}_n^{(\ell)}\|_{L^2(\R^2)}^2+\circ(1),\quad n\to\infty. \eq
\end{thm}
\begin{rems}\quad \\
\begin{itemize}
\item The example  by Moser reads as $$
f_{\alpha_n}(x)= \sqrt{\frac{\alpha_n}{2\pi}}\;{\mathbf L}\Big(\frac{-\log|x|}{\alpha_n}\Big),
$$
where
\begin{eqnarray*}
{\mathbf L}(s)&=&\; \left\{
\begin{array}{cllll}0 \quad&\mbox{if}&\quad
s\leq 0,\\ s \quad
&\mbox{if}&\quad 0\leq s\leq 1,\\
1 \quad&\mbox{if}&s\geq 1.
\end{array}
\right. \label{profilmos}
\end{eqnarray*}
 \smallbreak
\item It was shown in  \cite{H} that the sequence $( f_{\alpha_n})$  can be written under the form:
$$ f_{\alpha_n} (x)=  \wt {f_{\alpha_n} }(x) + {\rm r}_n (x),$$ 
with $\|\nabla
{\rm r}_n\|_{L^2} \stackrel{n\to\infty}\longrightarrow 0$  and 
\begin{equation} \label{logmoser}\wt {f_{\alpha_n} }(x) = \frac{1}{(2\pi)^2} \sqrt{\frac{2\pi}{\alpha_n}}\int_{ \R^2}  {\rm e}^{ i \, x \cdot \xi} \frac{1}{|\xi|^2}\,  {\varphi } \Big(\frac {\log |\xi| } {\alpha_n}\Big) d \xi\, ,\end{equation}
where $ {\varphi }(\eta) =   {\bf 1}_{[0, 1]} (\eta)$. 
    \end{itemize}
    \end{rems}

 \medbreak
 Let us point out that  the strategy  adopted in \cite{BMM1} to  build the asymptotic decomposition in terms of generalizations of the example by Moser concentrated around cores is based on capacity arguments and uses in a crucial way the fact that the  Schwarz symmetrization minimizes the energy (for  more details, we refer the reader to \cite{Kes, Analysis} and the references therein).  As the Schwarz symmetrization process does not allow to control the $H^N$-norm when $N> 1$, the method  developed in \cite{BMM1} does not apply to higher dimensions.\\
 
 \medbreak
The strategy we adopt here to describe
 the lack of compactness of the  Sobolev  embedding  \eqref{embedorliczNN} in the 2ND general case  is rather inspired by the approach of   P. G\'erard in  \cite{G} which is based on Fourier analysis. However, unlike the framework studied by P. ~G\'erard,  the elements which are responsible of the lack of compactness of \eqref{embedorliczNN} are spread in frequency. We are therefore led to revisit the method of P. G\'erard to address that issue. \\

 \medbreak
 Note that the description of the lack of compactness in other critical Sobolev embeddings was  achieved in \cite{BCG, BZ, G, Ja} and has been at the origin of several developments such as regularity results for Navier-Stokes systems in \cite{HG, gkp}, qualitative study of nonlinear evolution equations in \cite{BG,ker,La,tao} or estimate of the life span of focusing semi-linear dispersive evolution equations in \cite{km}. Further applications of the characterization of the lack of compactness in critical Sobolev embeddings are available in the elliptic framework. Among others, one can mention \cite{BC,Lions1,Lions2,Struwe88}.\\

\bigbreak 

 The main purpose of this paper is to establish the following theorem: 
 
\begin{thm}
\label{main2} Let $(u_n)$ be a bounded sequence in
$H^N(\R^{2N})$ such that \bq \label{main-assum1}
u_n\rightharpoonup 0  \quad
\mbox{in}  \quad H^N(\R^{2N}), \quad n\to\infty\,,\eq \bq \label{main-assum2}
\limsup_{n\to\infty}\|u_n\|_{\mathcal L}=A_0>0 \quad \quad
\mbox{and} \eq \bq \label{main-assum3} \lim_{R\to\infty}\;
\limsup_{n\to\infty}\,\int_{|x|\geq R}\,|u_n|^2\,dx=0\,. \eq Then,  there
exist a sequence of  scales $(\underline{\alpha}^{(j)})$, a sequence of cores $(\underline{x}^{(j)})$ and a sequence of profiles $( {\varphi^{(j)} })$  such that the triplets $(\underline{\alpha}^{(j)}, \underline{x}^{(j)}, {\varphi^{(j)} })$ are pairwise
orthogonal in the sense of Definition \ref{orthogen}  and such that, up to a subsequence extraction,  we have for all
$\ell\geq 1$, \bq \label{decompN}
u_n(x)=\Sum_{j=1}^{\ell}\,   \frac {C_{N}} { \sqrt{\alpha_n^{(j)}}}\int_{|\xi|\geq 1}\frac{ {\rm e}^{ i \, ( x-x^{(j)}_n)\cdot \xi}}{|\xi|^{2N}}   \; {\varphi^{(j)} }\Big(\frac {\log |\xi| } {\alpha^{(j)}_n}\Big) \, d \xi +{\rm
r}_n^{(\ell)}(x),\,  \eq with $\ds C_{N}= \frac {1} { (2\pi)^{N}  \sqrt{\omega_{2N-1}}}$ and $\ds \limsup_{n\to\infty}\;\|{\rm
r}_n^{(\ell)}\|_{\mathcal
L}\stackrel{\ell\to\infty}\longrightarrow 0$. \\ 

 \noindent Moreover for all
$\ell\geq 1$, we have 
the following stability estimates \bq \label{ortogonal} \|
u_n\|_{\dot H^N(\R^{2N})}^2=\Sum_{j=1}^{\ell}\, \|  {\varphi^{(j)} }\|_{L^2(\R_+)}^2+\|
{\rm r}_n^{(\ell)}\|_{\dot H^N(\R^{2N})}^2+\circ(1),\quad n\to\infty. \eq
\end{thm}

\medbreak 

The following proposition the proof of which is postponed to Appendix \ref{genemoser} allows to relate in the 2D case Theorem \ref{main2}  to Theorem \ref{noradmain}.  \begin{prop}
\label{genemoserN} Let  $(\alpha_n)_{n \geq 0 }$ be a  scale in the sense of Definition \ref{orthogen} and $\varphi$  in $L^2 (\R_+)$. 
Set with the notation of Theorem   \ref{main2}
$$ g_n(x):=  \frac {C_{N}} { \sqrt{\alpha_n}} \int_{|\xi| \geq {\rm
1}}\frac{ {\rm e}^{ i \, x \cdot \xi}}{|\xi|^{ 2 N }}  \; {\varphi} \Big(\frac { \log |\xi|} {\alpha_n}\Big) \, d \xi\,.$$ 
Then 
 \begin{equation}  \label{eqmoserN}g_n(x) = \wt C_N \,\sqrt{\alpha_n}\;\psi\Big(\frac{-\log|x|}{\alpha_n}\Big)\, + {\rm t}_n(x)\,,  \end{equation}
  with $\ds \psi(y)= \int^y_0  {\varphi}(t)  \, dt$, $\ds \wt C_N= \frac {\sqrt{ \omega_{2N-1}}} { (2\pi)^N}$  and $\|
{\rm t}_n\|_{\mathcal L} \stackrel{n\to\infty}\longrightarrow0$.
\end{prop}
By this proposition, one can write the elementary concentrations
\begin{equation}  \label{el} f^{(j)}_n(x):=  \frac {C_{N}} {  \sqrt{ \alpha^{(j)}_n}} \int_{|\xi| \geq {\rm
1}}\frac{ {\rm e}^{ i \, (x-x^{(j)}_n) \cdot \xi}}{|\xi|^{ 2 N }}  \; {\varphi^{(j)}} \Big(\frac { \log |\xi|} {\alpha^{(j)}_n}\Big) \, d \xi\,,\end{equation}
under the form \begin{equation}  \label{eqmoserNN}f^{(j)}_n(x)  = \wt C_N \,\sqrt{\alpha^{(j)}_n}\;\psi^{(j)}\Big(\frac{-\log|x-x^{(j)}_n|}{\alpha^{(j)}_n}\Big)+ {\rm t}^{(j)}_n(x),  \end{equation}
where $\ds \wt C_N= \frac {\sqrt{ \omega_{2N-1}}} { (2\pi)^N}$ and $\ds \|
{\rm t}^{(j)}_n\|_{\mathcal L} \stackrel{n\to\infty}\longrightarrow0$. This allows to view Theorem \ref{main2}  as a generalization of  Theorem \ref{noradmain}. 

\begin{rems}\quad \\
\begin{itemize}
\item Arguing exactly as in the proof of Proposition 1.15 in \cite{BMM} and  making use of    \eqref{Mos2} and  \eqref{eqmoserNN}, we can prove that the elementary concentrations $f^{(j)}_n$ satisfy 
\bq \label{profileN}
\Lim_{n\to\infty}\,\|f^{(j)}_n\|_{{\cL}}=\frac{1}{\sqrt{\beta_N}}\,\max_{s>0}\;\frac{|\psi^{(j)}(s)|}{\sqrt{s}}\, , \eq
where $\ds \psi^{(j)}(s):= \int^s_0  {\varphi^{(j)}}(t)  \, dt$.  \\ 
 \item Note that the elementary concentrations $f^{(j)}_n$ defined above by \eqref{el}  belong to $H^N(\R^{2N})$ while a priori the generalizations of the 
example by Moser $\ds \wt C_N \,\sqrt{\alpha^{(j)}_n}\;\psi^{(j)}\Big(\frac{-\log|x-x^{(j)}_n|}{\alpha^{(j)}_n}\Big)$ only belong to $H^1(\R^{2N})$.

 \smallbreak
 \item The hypothesis of compactness at infinity \eqref{main-assum3} is crucial: it allows to avoid the loss of Orlicz norm at infinity; without this assumption, the result is not true. 
 
  \smallbreak
 \item Note also that the lack of compactness of the Sobolev embedding \eqref{embedorliczNN}  was studied in the radial framework for the  4D case in \cite{BZ} by  tracking the large values of  the  sequences studied. 

 \end{itemize}
\end{rems}
 \subsection{Layout of the paper}

The  paper is organized as follows:   in Section \ref{decomposition} we recall some useful results  about profile decompositions.  The proof of Theorem \ref{main2} is addressed in Section \ref{new}. 
In Section \ref{Appendix}
 we  establish Proposition \ref{genemoserN} and highlight the connection between  Orlicz space and some space involved in the proof of Theorem \ref{main2}.

 \medskip

We mention that the letter~$C$ will be used to denote an absolute constant
which may vary from line to line.  We also use $A\lesssim B$ to
denote an estimate of the form $A\leq C B$ for some absolute
constant $C$. For simplicity, we shall also still denote by $(u_n)$ any
subsequence of $(u_n)$ and designate by $\circ(1)$ any sequence which tends to $ 0 $ as $n$ goes to infinity.  
 %%%%%%%%%%%%%%%%%%%%%%%%%%%%%%%%%%%%%

 \section{Background material}\label{decomposition}
  %%%%%%%%%%%%%%%%%%%%%%%%%%%%%%%%%%%%%

   Let us  start by introducing    the notions of being $\log $-oscillating with respect to a scale and of being $\log$-unrelated to any scale, which are a natural adaptation to our setting of the vocabulary of   P. G\'erard in  \cite{G}.
\begin{defi}
\label{oscilunreorlicz}
Let $v:= (v_n)_{n \geq 0}$ be a bounded sequence in $L^2(\R^d)$ and $\alpha:=(\alpha_n)_{n \geq 0}$ be a sequence of positive real numbers.
\begin{itemize}
\item The sequence $v$ is said $\alpha$ $\log $-oscillating if
\begin{equation} \label{oscilorlicz}\limsup_{n\to\infty}\; \left(\int_{|\xi|  \leq {\rm e}^{ \frac{\alpha_n} R }}  |\wh{ v_n}(\xi)|^2 \, d\xi + \int_{|\xi|  \geq {\rm e}^{ R \alpha_n}}  |\wh{ v_n}(\xi)|^2 \, d\xi\right) \stackrel{R\to\infty}\longrightarrow 0\,. \end{equation}
\item The sequence $v$ is said $\log$-unrelated to the scale $\alpha$ if for any real numbers $b >a>0$
\begin{equation} \label{unscalesorlicz}\int_{{\rm e}^{ a \alpha_n} \leq  |\xi|  \leq {\rm e}^{ b \alpha_n}}  |\wh{ v_n}(\xi)|^2 \, d\xi  \stackrel{n\to\infty}\longrightarrow 0\,. \end{equation}\end{itemize}
\end{defi}
\begin{rems}  \quad \\ \label{vocabulary}
\begin{itemize}
\item In what follows, it will be convenient for us to write in the 2ND case \footnote{where obviously  $ \xi=|\xi|\cdot \omega $, with $\omega \in {\mathbb  S}^{2N-1}$.} $$\wh{  v_n}(\xi) = \frac 1 {|\xi|^{N}}  \; { \varphi_n}(\log |\xi|, \omega).$$
 Note that if $(v_n)_{n \in \N}$ is a bounded sequence in $L^2(\R^{2N})$ whose Fourier transform is supported in  $\ds \big\{\xi \in \R^{ 2 N }; |\xi| \geq 1 \big\}$, then $(\varphi_n)_{n \in \N}$ is a bounded sequence in the space~$L^2 (\R_+ \times {\mathbb  S}^{2N-1})$. In terms of $(\varphi_n)$, the property for $(v_n)$ to be  $\alpha$ $\log $-oscillating can be written as: 
 $$\qquad \qquad \limsup_{n\to\infty}\big(\int_{{\mathbb  S}^{2N-1}} \int^{\frac {\alpha_n } {R }}_0 |\varphi_n(t,\omega)|^2  dt   d\omega+\int_{{\mathbb  S}^{2N-1}} \int^{\infty}_{{\alpha_n }R } |\varphi_n(t,\omega)|^2 dt  d\omega \big)\stackrel{R\to\infty}\longrightarrow 0\,.$$ 
 We will say that $(\varphi_n)_{n \in \N}$  is $\ds \frac 1 {\alpha_n }$-concentrated.
 \medbreak \noindent Similarly, the property for $(v_n)$  being   $\log $-unrelated to the scale $\alpha$ means that for any real numbers $b >a>0$
  $$   \int_{{\mathbb  S}^{2N-1}} \int^{b \, {\alpha_n }}_{a \, {\alpha_n }} |\varphi_n(t,\omega)|^2  \, dt  \, d\omega\stackrel{n\to\infty}\longrightarrow 0\,.$$ 
  The sequence  $(\varphi_n)_{n \in \N}$ will be said unrelated to the scale $\ds \frac 1 {\alpha_n }\cdot$ 
\item
 According to \eqref{logmoser}, we can easily prove that the example by Moser  $\nabla f_{\alpha_n}$ is $\alpha$ $\log $-oscillating. Indeed for $R \geq 1$\begin{eqnarray*}  \int_{ |\xi| \leq  {\rm e}^{  \frac {\alpha_n} R} }
 | \wh {\nabla f_{\alpha_n}}(\xi)|^2 \,d\xi &=& \int_{ |\xi| \leq  1}  | \wh {\nabla f_{\alpha_n}}(\xi)|^2 \,d\xi  + \int_{ 1\leq |\xi| \leq  {\rm e}^{  \frac {\alpha_n} R}}  | \wh {\nabla f_{\alpha_n}}(\xi)|^2 \,d\xi\,.\end{eqnarray*}
Firstly    $$ \int_{ |\xi| \leq  1}  | \wh {\nabla f_{\alpha_n}}(\xi)|^2 \,d\xi \leq \|f_{\alpha_n}\|_{L^2}^2\stackrel{n\to\infty}\longrightarrow 0\,,$$ and secondly thanks to \eqref{logmoser} \begin{eqnarray*} \int_{ 1\leq |\xi| \leq  {\rm e}^{  \frac {\alpha_n} R}}  | \wh {\nabla f_{\alpha_n}}(\xi)|^2 \,d\xi\leq  \frac{C}{\alpha_n} \int^{{\rm e}^{  \frac {\alpha_n} R}}_{ 1} \frac{dr}{r} + \circ(1)\leq  \frac{C}{R}+ \circ(1)\,,\end{eqnarray*}
 where $\circ(1)$ denotes
a sequence which tends to $ 0 $ as $n$ goes to infinity.  Therefore 
$$ \limsup_{n\to\infty}\; \int_{|\xi| \leq  {\rm e}^{  \frac {\alpha_n} R}}  |\wh {\nabla f_{\alpha_n}}(\xi)|^2 \, d\xi \stackrel{R\to\infty}\longrightarrow 0\,. $$
Finally Identity \eqref{logmoser} easily implies that 
$$ \limsup_{n\to\infty}\; \int_{|\xi| \geq  {\rm e}^{  R \alpha_n}}  |\wh {\nabla f_{\alpha_n}}(\xi)|^2 \, d\xi \stackrel{R\to\infty}\longrightarrow 0\,, $$
which  ends the proof of the result. \end{itemize}\end{rems}\noindent 

In the proof of Theorem \ref{main2}, we shall need  the following result which is an immediate consequence of Theorem 2.9 in \cite{G}.  Note that in that framework a scale  will designate a sequence $h:=(h_n)$
of positive real numbers and  ${\bf 1}$ the scale in which all the terms are equal to the number $1$.\begin{thm} \label{oscomp}
Let  $(\varphi_n)_{n \in \N}$ be a bounded sequence in $L^2 (\R_+ \times {\mathbb  S}^{2N-1})$. Then, there
exist a sequence of scales $(h^{(j)})_{j \geq 1}$   and a sequence of bounded sequences  $(g^{(j)})_{j \geq 1}$ in the space~$L^2 (\R_+ \times {\mathbb  S}^{2N-1})$ such that:
\begin{itemize}
\item if $j \neq k$,    $  \Big|\log\left({ h^{(j)}_n}/h^{(k)}_n\right)\Big|\stackrel{n\to\infty}\longrightarrow \infty$; 
\item for all $j$, $g^{(j)}$ is $h^{(j)}$-concentrated;
\item up to a subsequence extraction, we have for all
$\ell\geq 1$
 \bq \label{decompsc}
\varphi_n(t,\omega)=\Sum_{j=1}^{\ell}\,g_n^{(j)}(t,\omega)+ {\rm
r}_n^{(\ell)}(t,\omega),\quad \mbox{with} \quad \limsup_{n\to\infty}\; \| {\rm
r}_n^{(\ell)}\|_{\mathcal B} \stackrel{\ell\to\infty}\longrightarrow 0\,,\eq\end{itemize}
where  \bq \label{defB}\| {\rm
r}_n^{(\ell)}\|_{\mathcal B}:= \sup_{k\in\Z}\; \int_{{\mathbb  S}^{2N-1}}\int_{2^k \leq t \leq 2^{k+1}} \big|  {\rm
r}_n^{(\ell)}\big(t,\omega)\big|^2 \, d t \, d\omega\,.\eq
Furthermore  for any $\ell$, $({\rm
r}_n^{(\ell)})$ is unrelated to the scales $(h^{(j)})$ for  $j=1,\cdots,\ell$, and therefore \bq \label{ortogonall2} \|
\varphi_n\|_{L^2 (\R_+ \times {\mathbb  S}^{2N-1})}^2=\Sum_{j=1}^{\ell}\, \| g_n^{(j)}\|_{L^2 (\R_+ \times {\mathbb  S}^{2N-1})}^2+\|
{\rm r}_n^{(\ell)}\|_{L^2 (\R_+ \times {\mathbb  S}^{2N-1})}^2+\circ(1). \eq
\end{thm}
 \section{Proof of Theorem \ref{main2}}\label{new}
\subsection{Scheme of the proof} The proof  relies on a diagonal  extraction process and is done in three steps. In the first
step, we  extract the $\log$-oscillating components of the sequence $(u_n)$ we investigate. For that purpose, we  reduce the problem to the study of the bounded sequence $(\varphi_n)$ in $L^2 (\R_+\times {\mathbb  S}^{2N-1})$ expressed in terms of $(u_n)$ outside the low frequencies as follows: 
$$\wh{  u_n}(\xi) = \frac 1 {|\xi|^{2N}}  \; { \varphi_n}(\log |\xi|, \omega).$$
Then we extract the concentrated components of the sequence $(\varphi_n)$ applying Theorem~\ref{oscomp}. This allows to complete this step expressing  $(u_n)$ by means of $(\varphi_n)$. \\

\noindent
The second step is dedicated to the extraction of the cores and the profiles. It   consists firstly  in applying  with a slight modification the processes developed in \cite{G},  secondly in  replacing the profiles obtained by this method in $L^2 (\R_+\times {\mathbb  S}^{2N-1})$ by their average on the sphere, and finally  in translating the result in terms of $(u_n)$. \\

\noindent
In the third step, we establish that the remainder term  tends to zero in the Orlicz space as the number of functions in the sum and $n$  go to infinity. As we shall see in the sequel,  the remainder term includes two parts. A first part which comes from the decomposition on  $\log$-oscillating components,  and that we deal thanks to the connection between  $\mathcal B$ space defined by  \eqref{defB} and  Orlicz space. Regarding the second part, it comes from the error committed by replacing  the profiles obtained in the second step by their average on the sphere. We treat it through some suitable  estimates.

\subsection{Extraction of $\log$-oscillating components}   In view of Assumption \eqref{main-assum3}, the  sequence $(u_n)$  converges towards $0$ in $L^2(\R^{2N})$. Therefore,  for any fixed $M$
$$ \int_{|\xi| \leq M} (1+|\xi|^2)^N \big|  \wh{  u_n} (\xi)\big|^2 d\xi \leq (1+M^2)^{N} \| u_n\|^2_{L^2(\R^{2N})} \stackrel{n\to\infty}\longrightarrow 0.$$
Thus writing \footnote{we recall that  $ \wh{ \chi (D)f} (\xi)= \chi (\xi)\wh{ f} (\xi)$.}
$$ u_n= \chi (D) u_n + (1-\chi) (D) u_n,$$
where $\chi$ is a radial function  in~$\mathcal{D}(\R^{2N})$ equal to one in $B(0,1)$ and valued in  $[0,1]$, we infer that 
\begin{equation} \label{dec} u_n= \wt u_n + {\rm
r}_n^{1},\end{equation}  where  $(\wt u_n)$ is a bounded sequence in $H^N(\R^{2N})$ the Fourier transform of which  is supported in $\ds \big\{\xi \in \R^{ 2 N }; |\xi| \geq 1 \big\}$ and $ \| {\rm
r}_n^{1}\|_{H^N(\R^{2N})}\stackrel{n\to\infty}\longrightarrow 0$.   \\

Our starting point consists  in observing that 
  \bq \label{equiv}  \wh{  \wt  u_n}(\xi) = \frac 1 {|\xi|^{2N}}  \;{ \varphi_n}(\log |\xi|, \omega),
 \eq
with ${ \varphi_n}$  a bounded sequence of $L^2 (\R_+\times {\mathbb  S}^{2N-1})$.  Indeed, we have 
\begin{eqnarray*} \|\wh{ |D|^N \wt  u_n}\|^2_{L^2 (\R^{2N})} &=&  \int_{{\mathbb  S}^{2N-1}} d  \omega \, \int^\infty_1 \big| { \varphi_n}  (\log \rho, \omega )\big|^2 \frac {d \rho}  {\rho}   \\ &=&  \int_{{\mathbb  S}^{2N-1}} d  \omega \, \int^\infty_0\big| { \varphi_n} (u, \omega )\big|^2 du, \end{eqnarray*}
which ensures the result.  Consequently, one can write 
\begin{equation} \label{rel1} u_n (x)=  \frac {1} {(2\pi)^{2N}}\int_{|\xi| \geq 1}\frac{ {\rm e}^{ i \, x \cdot \xi}}{|\xi|^{2N}}  \; { \varphi_n}  (\log |\xi|, \omega) \, d \xi +{\rm
r}_n^{1}(x)\,,  \end{equation} with  $({ \varphi_n})_{n \in \N}$   a bounded sequence in $L^2 (\R_+\times {\mathbb  S}^{2N-1})$. Taking advantage of the fact that 
 $\|\wt  u_n\|_{L^2 (\R^{2N})} \stackrel{n\to\infty}\longrightarrow 0$, we infer with the vocabulary of Remarks  \ref{vocabulary} that the sequence $( \varphi_n)_{n \in \N}$ is unrelated to the  scale ${\bf 1}$ and  to any scale $( h_n)_{n \in \N}$ tending to infinity.  Indeed, for any real numbers $b >a>0$
\begin{eqnarray*} \int_{{\mathbb  S}^{2N-1}} d  \omega \,\int_{a \leq  u \leq b}  |{ \varphi_n}(u,\omega)|^2 \, du &\leq& {\rm
e}^{ 2 N b} \int_{{\mathbb  S}^{2N-1}} d  \omega \,\int_{a \leq  u  \leq b}  |{ \varphi_n}(u,\omega)|^2 {\rm
e}^{- 2 N u}\, du \\ &\leq& {\rm
e}^{ 2 N b} \|\wt  u_n\|^2_{L^2 (\R^{ 2 N })}\leq {\rm
e}^{ 2 N b} \|  u_n\|^2_{L^2 (\R^{ 2 N })}\stackrel{n\to\infty}\longrightarrow 0.\end{eqnarray*}
Along the same lines, if $( h_n)_{n \in \N}$ is any scale  tending to infinity, then we have 
for any real numbers $b >a>0$ and $n$ large enough
\begin{eqnarray*} \int_{{\mathbb  S}^{2N-1}} d  \omega \,\int_{a \leq h_n u \leq b}  |{ \varphi_n}(u,\omega)|^2 \, du &\leq& {\rm
e}^{ 2 N b} \int_{{\mathbb  S}^{2N-1}} d  \omega \,\int_{0 \leq  u  \leq b}  |{ \varphi_n}(u,\omega)|^2 {\rm
e}^{- 2 N u}\, du \,,  \end{eqnarray*}
which  ensures the result.  \\

 \noindent Now to extract the  $\log$-oscillating components  of the sequence $(u_n)$,  we shall apply Theorem \ref{oscomp}    to  the sequence $(\varphi_n)$.  Up to a subsequence extraction, this gives rise to   
\bq \label{apger}\varphi_n(t,\omega)=\Sum_{j=1}^{\ell}\, \wt {g_n}^{(j)}(t,\omega)+\wt{\rm
t}_n^{(\ell)}(t,\omega), \eq
for $(t,\omega) \in \R_+\times {\mathbb  S}^{2N-1}$, where  the sequence  $(\wt {g_n}^{(j)})$ is $h_n^{(j)}$-concentrated, with for $j \neq k$ $\big|\log\big({ h^{(j)}_n}/h^{(k)}_n\big)\big|\stackrel{n\to\infty}\longrightarrow \infty$,   and where $({\wt{\rm
t}_n^{(\ell)}})$ is unrelated to any scale $(h_n^{(j)})$ for  $j=1,\cdots,\ell$,  and satisfies $$\limsup_{n\to\infty}\; \sup_{j\in\Z}\; \int_{{\mathbb  S}^{2N-1}} d  \omega \,\int_{2^j}^{2^{j+1}}\big|  {\wt{\rm
t}_n^{(\ell)}}\big(t,\omega \big)\big|^2 \, d t \stackrel{\ell\to\infty}\longrightarrow 0.$$ Taking advantage of the above, we deduce that the  components $\wt {g_n}^{(j)}$  intervening in Decomposition \eqref{apger} are $h_n^{(j)}$-concentrated  with $h_n^{(j)} \stackrel{n\to\infty}\longrightarrow 0$. \\

Therefore,  up to a subsequence extraction, we have 
 \bq \label{first}  u_n (x) = \Sum_{j=1}^{\ell}\, g_n^{(j)}(x)+{\rm
t}_n^{(\ell)}(x) + \circ(1)\,, \eq
where the Fourier transforms   of the sequences  $ ({\rm
t}_n^{(\ell)})$ and  $(g_n^{(j)})$ for $j \in \{1,\cdots,\ell\}$ are  supported in  $\ds \big\{\xi \in \R^{ 2 N }; |\xi| \geq 1 \big\}$ and satisfy 
 \bq \label{passage} \wh {g_n^{(j)}}(\xi) =  \frac 1 {|\xi|^{ 2 N }}  \; {\wt {g_n}^{(j)}}(\log |\xi|,\omega) \quad  \mbox{and} \quad  \wh {{\rm
t}_n^{(\ell)}}(\xi) =  \frac 1 {|\xi|^{ 2 N }}  {\wt {\rm
t}_n^{(\ell)}}(\log |\xi|,\omega)\,, \eq
and where $ \circ(1)$   tends to $ 0 $  in $H^N(\R^{2N})$ and thus in ${\mathcal
L}$ as $n$ goes to infinity.  Moreover, in light of   Proposition \ref{uspprop*} which relates the Orlicz and ${\mathcal B}$ norms  \begin{equation}   \limsup_{n\to\infty}\;\|{\rm
t}_n^{(\ell)}\|_{\mathcal L} \stackrel{\ell\to\infty}\longrightarrow 0\label{rel3}.\end{equation}

\noindent It is then obvious that  the sequences $(|D|^N g_n^{(j)})$ are $\alpha^{(j)}$ $\log $-oscillating in the sense of Definition \ref{oscilunreorlicz}, with $ \alpha^{(j)}_n = \frac 1  {h^{(j)}_n} \stackrel{n\to\infty}\longrightarrow \infty\,,$ and that $(|D|^N {\rm
t}_n^{(\ell)})$  is $\log$-unrelated to the scales $\underline{\alpha}^{(j)}$ for $j=1,\cdots,\ell$. Indeed, 
\begin{eqnarray*} \int_{\frac{1}{\alpha^{(j)}_n} (\log |\xi|) \geq R}  | |\xi|^N \, \wh {g_n^{(j)}}(\xi)|^2 \, d\xi &=&   \int_{{\mathbb  S}^{2N-1}} d  \omega \, \int_{\frac{1}{\alpha^{(j)}_n} (\log \rho ) \geq R}  | {\wt {g_n}^{(j)}}(\log \rho,\omega)|^2 \, \frac {d \rho}  {\rho} \\ &=&  \int_{{\mathbb  S}^{2N-1}} d  \omega \, \int_{\frac{\eta  }{\alpha^{(j)}_n}  \geq R}  | {\wt {g_n}^{(j)}}(\eta,\omega)|^2 \, d \eta\,. \end{eqnarray*}
The sequence ${\wt {g_n}^{(j)}}$  being  $\frac 1  {\alpha^{(j)}_n}$-concentrated, we deduce that  $$\limsup_{n\to\infty}\;  \int_{\frac{1}{\alpha^{(j)}_n} (\log |\xi|) \geq R}  | |\xi|^N \,  \wh{g_n^{(j)}}(\xi)|^2 \, d\xi   \stackrel{R\to\infty}\longrightarrow 0\,.$$
Along the same lines, we get 
$$\limsup_{n\to\infty}\;  \int_{\frac{1}{\alpha^{(j)}_n} (\log |\xi|) \leq \frac{1} R}  | |\xi|^N \,  \wh {g_n^{(j)}}(\xi)|^2 \, d\xi   \stackrel{R\to\infty}\longrightarrow 0\,,$$ and establish that $(|D|^N {\rm
t}_n^{(\ell)})$  is $\log$-unrelated to the scales $\underline{\alpha}^{(j)}$ for $j=1,\cdots,\ell$. This  achieves the proof of 
the claim. The fact that  $ \alpha^{(j)}_n = \frac 1  {h^{(j)}_n} $ easily implies that the  scales $(\underline{\alpha}^{(j)})$ are  pairwise
orthogonal  in the sense of \eqref{caseI}.
 
\subsection{Extraction of  cores and profiles}    
Now  our aim is to prove that the sequence $(g_n^{(j)})_{n \in \N }$ 
can be decomposed up to a subsequence extraction according to 
\begin{equation}  \label{hope} \sum_{l=1}^L    \frac {C_{N}} { \sqrt{\alpha_n^{(j)}}}\int_{\R^{ 2 N }}\frac{ {\rm e}^{ i \, (x-x^l_n) \cdot \xi}}{|\xi|^{ 2 N }}   \; {\phi^l} \Big(\frac {\log |\xi| } {\alpha^{(j)}_n} \Big) \, d \xi+\circ(1)\,, \end{equation}
where the profiles $ {\phi^l}$ belong to $ L^2(\R_+)$ and  the cores $(x_n^l)_{n \in \N }$ are sequences of points in $\R^{ 2 N }$ satisfying the orthogonality assumption \eqref{caseII},  and  where $ \circ(1)$  designates
a sequence which tends to $ 0 $ in ${\mathcal
L}$ as $n$ and $L$ go to infinity. \\

\noindent
For that purpose, let us firstly recall   that 
\begin{eqnarray*} g_n^{(j)}(x)  &=&\frac {1} { (2\pi)^{ 2 N }} \,  \int_{|\xi| \geq 1}\frac{ {\rm e}^{ i \, x \cdot \xi}}{|\xi|^{ 2 N }}   \; {\wt {g_n}^{(j)}}(\log |\xi|,\omega)\, d \xi,\end{eqnarray*}
where the sequence  $({\wt {g_n}^{(j)}} )$  is  $   \ds \frac 1  {\alpha^{(j)}_n}$-concentrated. Thus   
\begin{equation}  \label{start}  g_n^{(j)}(x) =    \frac {C_{N}} { \sqrt{\alpha_n^{(j)}}} \int_{|\xi| \geq 1}\frac{ {\rm e}^{ i \, x \cdot \xi}}{|\xi|^{ 2 N }}  \; {\varphi_n^{(j)}} \Big(\frac { \log |\xi|} {\alpha^{(j)}_n},\omega \Big) \, d \xi\,,\end{equation}
with $\varphi_n^{(j)}$ a bounded sequence in $L^2(\R_+\times {\mathbb  S}^{2N-1})$ which is  ${\bf 1}$-concentrated. \\

\noindent In order to establish  \eqref{hope}, we shall revisit the proof of Proposition 4.1 of   P. G\'erard  in \cite{G}. For that purpose, with a fixed scale $(\alpha_n)$, let us for a  sequence $(\varphi_n)$  bounded  in the space~$L^2(\R_+\times {\mathbb  S}^{2N-1})$ and  ${\bf 1}$-concentrated  denote  by ${\cP}(\varphi_n)$  the set of weak  limits  in $ L^2(\R_+\times {\mathbb  S}^{2N-1})$ of subsequences of 
$$  {\rm e}^{ i \, x_n \cdot  {\rm e}^{\alpha_n t}\omega} \;\;  {\varphi_n}(t,\omega),$$ 
where $x_n \in \R^{ 2 N }$ 
and set 
\begin{equation}  \label{defeta} {\mathcal \eta}( \varphi_n):= \sup_{\varphi \in {\cP}(\varphi_n)} \|\varphi\|^2_{L^2(\R_+\times {\mathbb  S}^{2N-1})}.\end{equation}
Now the heart of the 
matter consists in  establishing the  following lemma:
 \begin{lem}
\label{lem1}
Let $(\varphi_n)$ be a bounded sequence in $L^2(\R_+\times {\mathbb  S}^{2N-1})$, ${\bf 1}$-concentrated, and  $(\alpha_n)_{n \in \N}$ be  a sequence
of positive real numbers going to infinity. Then, up to a subsequence extraction,  there exist a sequence $( \Phi^{(l)} )_ {l \geq 1 }$ of functions belonging to  $ L^2(\R_+\times {\mathbb  S}^{2N-1})$ and  a sequence of cores $(\underline{x}^{(l)})_ {l \geq 1}$,   such that  for any $L\geq 1$  we have 
$$ \varphi_n (t, \omega)= \Sum_{l=1}^{L}  {\rm e}^{- i \, x^{(l)}_n \cdot  {\rm e}^{\alpha_n t}\omega} \;\Phi^{(l)}(t,\omega) +  {\rm r}^{(L)}_n(t,\omega)\,,$$ 
with
 \begin{enumerate}
    \item For all $1\leq  l'< l $, we have
  $$ - \frac{ \log|x^{(l)}_n- x^{(l^\prime)}_n|}{\alpha_n} \stackrel{n\to\infty}\longrightarrow a \in [-\infty,+\infty[  \quad \mbox{with in the case when}  \quad a > -\infty,  \,\,$$
$$  \Phi^{(l)}(t,\omega) = 0,  \quad \mbox{for} \quad t < a  \,.$$
  \item $\ds  \qquad  \qquad {\rm e}^{ i \, x^{(L)}_n \cdot  {\rm e}^{\alpha_n t}\omega}{\rm
r}_n^{(L)}(t,\omega) \; \stackrel{n\to\infty}\rightharpoonup 0  \quad \mbox{in} \quad L^2(\R_+\times {\mathbb  S}^{2N-1})\,. $\\
\item  For all $1\leq  l'< l $, the sequence 
$$  \qquad  \qquad {\rm e}^{ i \,( x^{(l)}_n- x^{(l')}_n) \cdot  {\rm e}^{\alpha_n t}\omega} \Phi^{(l)}(t,\omega) \; \rightharpoonup 0  \quad \mbox{in} \quad L^2(\R_+\times {\mathbb  S}^{2N-1}),\quad n\to\infty\,.$$
  \item $ \ds \qquad \quad \limsup_{L\to\infty} \;{\mathcal \eta}({\rm
r}_n^{(L)}) = 0.$
 \end{enumerate}
 \end{lem}
 \begin{rem}
Let us point out that  Claims (2) and (3) ensure that for all $l\leq L$, 
 ${\rm e}^{ i \, x^{(l)}_n \cdot  {\rm e}^{\alpha_n t}\omega}{\rm
r}_n^{(L)}(t,\omega) \; \stackrel{n\to\infty}\rightharpoonup 0  \quad \mbox{in} \quad L^2(\R_+\times {\mathbb  S}^{2N-1})\,$ and in particular,
   \begin{equation}   \label{ort*}\|
\theta \, \varphi_n \|_{L^2}^2=\Sum_{l=1}^{L}\, \| \theta \, \Phi^{(l)}\|_{L^2}^2+\|
\theta \,{\rm r}^{(L)}_n\|_{L^2}^2+\circ(1),\quad n\to\infty\,,  \end{equation}
for any  $\theta$ in  $ L^\infty(\R_+)$.
  \end{rem}
 \begin{proof} 
If $\ds {\mathcal \eta}(\varphi_n) =0$, we are done. Otherwise  there exists $ \ds  \Phi^{(1)} \in {\cP}(\varphi_n)$ such that 
$$  \|  \Phi^{(1)}\|_{L^2}^2 \geq \frac 1 2 \, {\mathcal \eta}(\varphi_n)\,. $$ 
Thus we have, up to a subsequence extraction
$$ \varphi_n(t,\omega)= {\rm e}^{- i \, x^{(1)}_n \cdot  {\rm e}^{\alpha_n t}\omega} \;\Phi^{(1)}(t,\omega) +  {\rm r}^{(1)}_n(t,\omega)\,,$$ 
with $(x^{(1)}_n)$  a core and 
\begin{equation}  \label{us*} {\rm e}^{ i \, x^{(1)}_n \cdot  {\rm e}^{\alpha_n t}\omega} \;  {\rm r}^{(1)}_n(t,\omega) \stackrel{n\to\infty}\rightharpoonup 0  \quad \mbox{in} \quad L^2(\R_+\times {\mathbb  S}^{2N-1})\,.\end{equation}
 Now  if $\ds {\mathcal \eta}({\rm r}^{(1)}_n) =0$, we stop the process. If not there exists $ \Phi^{(2)} \in {\cP}(\varphi_n)$ such that 
$$  \|\Phi^{(2)}\|_{L^2}^2 \geq \frac 1 2\,{\mathcal \eta}( {\rm r}^{(1)}_n)\,,$$ 
which, up to a subsequence extraction, gives rise to 
$$ {\rm r}^{(1)}_n(t,\omega)= {\rm e}^{- i \, x^{(2)}_n \cdot  {\rm e}^{\alpha_n t}\omega} \;\Phi^{(2)}(t,\omega) +  {\rm r}^{(2)}_n(t,\omega)\,,$$ 
for some core $(x^{(2)}_n)$,  and with
$$ {\rm e}^{ i \, x^{(2)}_n \cdot  {\rm e}^{\alpha_n t}\omega} \;  {\rm r}^{(2)}_n(t,\omega) \stackrel{n\to\infty}\rightharpoonup 0  \quad \mbox{in} \quad L^2(\R_+\times {\mathbb  S}^{2N-1})\,.$$
 Up to a subsequence extraction, we can suppose that 
 $$- \frac{ \log|x^{(1)}_n- x^{(2)}_n|}{\alpha_n} \stackrel{n\to\infty}\longrightarrow a \in [-\infty,+\infty]\,.$$ We claim in the case when $a > -\infty$ that  $  \Phi^{(2)}(t)=0$ for all $t<a$. Indeed, if not there exists $\wt a < a$ such that 
\begin{equation}  \label{abs1} \int^{\wt a}_ {-\infty}  \int_{{\mathbb  S}^{2N-1}}\big|  \Phi^{(2)}(t,\omega)\big|^2 \, d t  \, d  \omega = \delta >0\,.\end{equation}
By hypothesis if $\epsilon >0$ is chosen so that $\wt a  + \epsilon < a$, then 
\begin{equation}  \label{us1}  - \frac{ \log|x^{(1)}_{n}- x^{(2)}_{n}|}{\alpha_{n}} \geq \wt a  + \epsilon, \quad \mbox{for} \quad n \quad \mbox{sufficiently large} \,.\end{equation}
\noindent
 But  by construction, we have 
 $$ \int^{\wt a}_ {-\infty}  \int_{{\mathbb  S}^{2N-1}}\big|  \Phi^{(2)}(t,\omega)\big|^2 \, d t  \,d  \omega = \lim_{n\to\infty}\; {\mathcal
I}_{n}\,,$$ 
 with
 $${\mathcal
I}_{n} :=  \int^{+\infty}_ {-\infty} \int_{{\mathbb  S}^{2N-1}} {\rm e}^{i \, x^{(2)}_{n} \cdot  {\rm e}^{\alpha_{n} t}\omega} {\rm r}^{(1)}_{n}(t, \omega )\; \chi_{[0,\wt a]}(t) \overline { \Phi^{(2)}}(t,\omega)\, d t  \, d  \omega\,.$$ 
Moreover \begin{eqnarray*}
{\mathcal
I}_{n}&=&  \int^{+\infty}_ {-\infty}  \int_{{\mathbb  S}^{2N-1}}{\rm e}^{i \, x^{(1)}_{n} \cdot  {\rm e}^{\alpha_{n} t}\omega} {\rm r}^{(1)}_{n}(t, \omega )\;  {\rm e}^{i \, (x^{(2)}_{n}-x^{(1)}_{n}) \cdot  {\rm e}^{\alpha_{n} t}\omega}\chi_{[0,\wt a]}(t) \overline {\ \Phi^{(2)}}(t,\omega)\, d t  \,d  \omega \\ &=&  \int^{+\infty}_ {-\infty}  \int_{{\mathbb  S}^{2N-1}} {\rm e}^{i \, x^{(1)}_{n} \cdot  {\rm e}^{\alpha_{n} t}\omega} {\rm r}^{(1)}_{n}(t, \omega )\;  \chi_{[0,\wt a]}(t) \overline { \Phi^{(2)}}(t,\omega)\, d t  \,d  \omega + {\mathcal
R}_n\,,
  \end{eqnarray*}
  with
  \begin{eqnarray*}
\big| {\mathcal
R}_{n}\big|&\leq &  \int^{\wt a}_ {-\infty}  \int_{{\mathbb  S}^{2N-1}}\big| {\rm r}^{(1)}_{n}(t, \omega )\big|\;  \big| \Phi^{(2)}(t,\omega)\big| \, \big|x^{(2)}_{n}-x^{(1)}_{n}\big| \,  {\rm e}^{\alpha_{n} t} d t  \,d  \omega \\ &\lesssim&  \big|x^{(2)}_{n}-x^{(1)}_{n}\big| \,  {\rm e}^{\alpha_{n} \wt a}\,.
  \end{eqnarray*} 
By virtue of \eqref{us1}
$$ \big|x^{(1)}_{n}- x^{(2)}_{n}\big| \leq {\rm e}^{- \, \alpha_{n} (\wt a  + \epsilon)}, \quad \mbox{for} \quad n \quad \mbox{sufficiently large},$$
which   implies that
$$\big| {\mathcal
R}_{n}\big| \lesssim {\rm e}^{ \,- \alpha_{n}  \epsilon} \stackrel{n\to\infty}\longrightarrow 0\,.$$
This in light of  \eqref{us*} yields a contradiction with \eqref{abs1} and  then concludes the proof of the orthogonality property for $\Phi^{(2)}$.  According to the fact that by construction $\Phi^{(2)}$ is not null, we deduce that $a< +\infty$.\\

\noindent We next address Claim (3), which is in fact a direct consequence of Claim (2).
Indeed,  
in view of the orthogonality property for $\Phi^{(2)}$, it suffices to demonstrate  in the case when $a < +\infty$ that  for all  $\ds g$ in  $ {\mathcal
D} (]a,+\infty[ \times {\mathbb  S}^{2N-1})$ 
  \begin{equation}   \label{comp}{\mathcal
J}_{n} :=  \int^{+\infty}_ {a}  \int_{{\mathbb  S}^{2N-1}}  {\rm e}^{i \, (x^{(1)}_{n}-x^{(2)}_{n}) \cdot  {\rm e}^{\alpha_{n} t}\omega} \,g(t,\omega)\, d t  \,d  \omega\stackrel{n\to\infty}\longrightarrow 0\,.\end{equation}
\noindent
For that purpose,  we  perform the change of variables $\xi= {\rm e}^{\alpha_n t}\cdot \omega$ which leads to   
$$
 {\mathcal
J}_{n}  =  \frac 1 {\alpha_{n}}   \int_{\R^{2N}}  {\rm e}^{i \, (x^{(1)}_{n}-x^{(2)}_{n}) \cdot \xi} \,g\Big(\frac {\log |\xi| } {\alpha_n},\frac { \xi} {|\xi|} \Big)\, \frac {d \xi} {|\xi|^{2N}} \,,$$
   where $\ds  g\Big(\frac {\log |\xi| } {\alpha_n},\frac { \xi} {|\xi|} \Big)$ is supported in $\ds \Big\{ {\rm e}^{\alpha_n a_1}\leq |\xi| \leq {\rm e}^{\alpha_n b_1} \Big\} $, with $a <  a_1 < b_1 < \infty$. \\ 

 \noindent
Then integrating by parts, we get  
   $$
 {\mathcal
J}_{n}  =  \frac i {\alpha_{n}}   \int_{\R^{2N}}  {\rm e}^{i \, (x^{(1)}_{n}-x^{(2)}_{n}) \cdot \xi} \, \frac {x^{(1)}_{n}-x^{(2)}_{n}} {|x^{(1)}_{n}-x^{(2)}_{n}|^{2}}\cdot \nabla_\xi \, \left[  \frac {1} {|\xi|^{2N}}g\Big(\frac {\log |\xi| } {\alpha_n},\frac { \xi} {|\xi|} \Big)\right] \, d \xi\,,$$
which implies that
 \begin{eqnarray*}
\Big|  {\mathcal
J}_{n}\Big| & \lesssim & \frac 1 {\alpha_{n}} \int^{ {\rm e}^{\alpha_n b_1}}_{ {\rm e}^{\alpha_n a_1}} \frac {1} {|x^{(1)}_{n}-x^{(2)}_{n}|} \, \frac {\rho^{2N-1}} {\rho^{2N+1} }\, d \rho \\ & \lesssim &  \frac {{\rm e}^{-\alpha_n a_1}} {\alpha_{n} \, |x^{(1)}_{n}-x^{(2)}_{n}|} \,\stackrel{n\to\infty}\longrightarrow 0\,,
    \end{eqnarray*} 
   and ensures the result. \\ 

 \noindent
 An  iteration  argument  allows to construct the families   $(\Phi^{(l)} )_ {l \geq 1 }$  and  $(\underline{x}^{(l)})_ {l \geq 1}$ satisfying  Claims $(1)$, $(2)$ and $(3)$   of Lemma \ref{lem1}. Finally, recalling that by  construction  $${\mathcal \eta}({\rm
r}_n^{(L)}) \leq 2 \, \| \Phi^{(L+1)}\|_{L^2}^2\,,$$  we deduce in view of  the convergence of the series $\Sum_{l \in \N} \| \Phi^{(l)}\|_{L^2}^2$ that  
$$ \ds \qquad \quad \limsup_{L\to\infty} \;{\mathcal \eta}({\rm
r}_n^{(L)}) = 0\,.$$
This ends the proof of the lemma. 
  \end{proof}
   
   \medbreak
   
   \begin{cor}  Under  notations of Lemma \ref{lem1}, we have 
   \begin{equation}  \label{heart} \limsup_{n\to\infty}\;\|R^{(L)}_n\|_{\mathcal
L}\stackrel{L\to\infty}\longrightarrow 0\,,\end{equation}
 where   $$ \ds R_n^{(L)}(x):=  
  \frac {C_{N}} { \sqrt{ \alpha_n}} \int_{|\xi| \geq 1}\frac{ {\rm e}^{ i \, x \cdot \xi}}{|\xi|^{ 2 N }}  \; {{\rm
r}_n^{(L)} } \Big(\frac { \log |\xi|} {\alpha_n},\omega \Big) \, d \xi.$$
   \end{cor}  
     \begin{proof}
   The goal is  to prove that  for all  $\varepsilon > 0$, there exist  $L_0 \in \N$ and $C$ an absolute constant such that
$$ \forall L \geq L_0,  \quad \limsup_{n\to\infty}\; \|R^{(L)}_n\|_{\mathcal L} \leq C\, \varepsilon\,.$$
Recall that  $\varphi_n$ is a bounded sequence in $L^2(\R_+\times {\mathbb  S}^{2N-1})$ which is  ${\bf 1}$-concentrated, and therefore  there exists $R_0>0$  such that 
\begin{equation}
\label{uskey} \limsup_{n\to\infty}\; \Big( \int_{{\mathbb  S}^{2N-1}} \int^{\frac 1 {R_0 }}_0 |\varphi_n(t,\omega)|^2  \, dt  \, d\omega+\int_{{\mathbb  S}^{2N-1}} \int^{\infty}_{R_0 } |\varphi_n(t,\omega)|^2  \, dt  \, d\omega \Big)\leq \varepsilon^2\,. \end{equation} 
Since by Assertion  \eqref{ort*}, we have for any $L \geq 1$
$$  \int_{{\mathbb  S}^{2N-1}} \int_{\{t \leq \frac 1 {R_0 }\} \cup \{t \geq R_0\}} | {\rm r}^{(L)}_n(t,\omega)|^2  \, dt  \, d\omega  \leq  \int_{{\mathbb  S}^{2N-1}} \int_{\{t \leq \frac 1 {R_0 }\} \cup \{t \geq R_0\}} |\varphi_n(t,\omega)|^2  \, dt  \, d\omega +\circ(1),$$ 
as $n$ tends to infinity, we obtain for any $L \geq 1$
\begin{equation}  \label{eqeq} \limsup_{n\to\infty}\;\int_{{\mathbb  S}^{2N-1}} \int_{\{t \leq \frac 1 {R_0 }\} \cup \{t \geq R_0\}} | {\rm r}^{(L)}_n(t,\omega)|^2  \, dt  \, d\omega  \leq  \varepsilon^2\,.\end{equation}
Now let us decompose $R^{(L)}_n$ into two parts as follows:  
$$R^{(L)}_n= \wt R^{(L)}_n + (R^{(L)}_n-\wt {R}^{(L)}_n)\,, $$  
 with
 $$\wt R^{(L)}_n(x):=   \frac {C_{N}} { \sqrt{ \alpha_n}} \int_{|\xi| \geq 1 }\frac{ {\rm e}^{ i \, x \cdot \xi}}{|\xi|^{ 2 N }} \, \chi_{R_0 }\Big(\frac { \log |\xi|} {\alpha_n} \Big)\,{  {\rm r}^{(L)}_n}\Big(\frac { \log |\xi|} {\alpha_n},\omega \Big)\, d \xi\,$$
and  where $\chi_{R_0 }$ is a  function  of~$\mathcal{D}(\R)$ which is equal to one in $\ds \Big\{\frac 1 {R_0 } \leq y \leq R_0\Big\}$, supported in~$\ds \Big\{\frac 1 {2 R_0 } \leq y \leq 2 R_0\Big\}$ and valued in  $[0,1]$. \\
 
 \noindent First let us  consider $\ds R^{(L)}_n-\wt {R}^{(L)}_n$.  By straightforward computations, we get 
$$ \|R^{(L)}_n-\wt {R}^{(L)}_n\|_{\mathcal L} \lesssim \|R^{(L)}_n-\wt {R}^{(L)}_n\|_{H^{N}(\R^{2N})} \lesssim \| \big(1-\chi_{R_0 }\big)  {\rm r}^{(L)}_n\|_{L^{2}}\,.$$  
Since
$$ \| \big(1-\chi_{R_0 }\big)  {\rm r}^{(L)}_n\|^2_{L^{2}}  \leq \int_{\{t \leq \frac 1 {R_0 }\} \cup \{t \geq R_0\}} | {\rm r}^{(L)}_n(t,\omega)|^2  \, dt  \, d\omega\,, $$ 
we deduce in light of \eqref{eqeq} that 
   $$  \limsup_{n\to\infty}\; \|R^{(L)}_n-\wt {R}^{(L)}_n\|_{\mathcal L}\lesssim \varepsilon. $$ 
Now let us  address $ \wt R^{(L)}_n$. To go to this end, let us  start by observing that  for any $L  \geq 1$
\begin{equation}  \label{est**}
 \limsup_{n\to\infty}\;  \frac 1 {\sqrt{\alpha_n} }  \,\| \wt R^{(L)}_n\|_{L^\infty(\R^{2N})} \leq {\mathcal \eta}( {\rm r}^{(L)}_n)^{\frac 1 2} \| \chi_{R_0 }\|_{L^2}\,.\end{equation}
Indeed, if not there exist an integer $L_0$,  a positive real number $\delta $  and  a subsequence $(\wt R^{(L_0)}_{n_k})$ such that 
$$ \frac 1 {\sqrt{\alpha_{n_k}} }  \,\| \wt R^{(L_0)}_{n_k}\|_{L^\infty(\R^{2N})} \geq {\mathcal \eta}( {\rm r}^{(L_0)}_{n_k})^{\frac 1 2} \| \chi_{R_0 }\|_{L^2} + \delta\,,$$
for $k$ sufficiently large.  Therefore, there exists a sequence of points $(x_{n_k})$ such that $$\frac 1 {\sqrt{\alpha_{n_k}} }  \, | \wt R^{(L_0)}_{n_k} (x_{n_k})| \geq  {\mathcal \eta}({\rm r}^{(L_0)}_{n_k})^{\frac 1 2}\,\| \chi_{R_0 }\|_{L^2} + \frac \delta 2\,,$$
which means that 
$$  \Big| \int_{\R_+} \int_{{\mathbb  S}^{2N-1}} \chi_{R_0 } (t)    {\rm e}^{ i \, x_{n_k} \cdot  {\rm e}^{\alpha_{n_k} t}\omega}   \, { {\rm r}^{(L_0)}_{n_k}} \Big(t,\omega \Big)\, dt \, d \omega\,\Big|  \geq  {\mathcal \eta}({\rm r}^{(L_0)}_{n_k})^{\frac 1 2}\,\| \chi_{R_0 }\|_{L^2} + \frac \delta 2\,,$$ for $k$  big enough.
But the sequence $( \ds {\rm e}^{ i \, x_{n_k} \cdot  {\rm e}^{\alpha_{n_k} t}\omega}   \,  { {\rm r}^{(L_0)}_{n_k}})$ is bounded in $L^2(\R_+\times {\mathbb  S}^{2N-1})$. Thus, up to a subsequence extraction, it converges weakly in $L^2$ to a function $H$ belonging to  ${\cP}( {\rm r}^{(L_0)}_{n_k})$. Passing to the limit, we deduce that 
$$ {\mathcal \eta}({\rm r}^{(L_0)}_{n_k})^{\frac 1 2}\,\| \chi_{R_0 }\|_{L^2} + \frac \delta 2 \leq \Big| \int_{\R_+} \int_{{\mathbb  S}^{2N-1}} \chi_{R_0 } (t)   H \Big(t,\omega \Big)\, dt \, d \omega\,\Big|  \leq \| H\|_{L^2} \| \chi_{R_0 }\|_{L^2}\,,$$  
which contradicts the definition of ${\mathcal \eta}( {\rm r}^{(L_0)}_{n_k})$ and ends the proof of Claim \eqref{est**}. \\
 
\noindent Now using the simple fact
that for any  function $u$
$$  \int_{\R^{2N}}\;\Big({\rm
e}^{|u(x)|^2}-1\Big)\,dx\, \lesssim  \int_{\R^{2N}}\;{\rm
e}^{|u(x)|^2}\,|u(x)|^2 \, dx,$$ 
we get for any $ \lambda > 0$ 
\begin{equation}  \label{est***}
  \int_{\R^{2N}}\;\Big({\rm
e}^{|\frac{\wt  R^{(L)} _n(x)}{\lambda}|^2}-1\Big)\,dx\, \lesssim \frac{1} {\lambda^2}\, {\rm e}^{\frac{\|\wt  R^{(L)} _n\|^2_{L^\infty}}{\lambda^2}} \|\wt  R^{(L)} _n\|^2_{L^2(\R^{2N})}\, .\end{equation} 
But 
$$ \|\wt  R^{(L)} _n\|^2_{L^2(\R^{2N})} \lesssim \int_{{\mathbb  S}^{2N-1}} \int^{2 R_0 }_{\frac 1 {2R_0 }} {\rm e}^{-2N  \alpha_n t}\, | { {\rm r}^{(L)}_n}(t,\omega)|^2  \, dt \, d\omega \lesssim  {\rm e}^{- \frac{ N  \alpha_n} {R_0}}\, \| { {\rm r}^{(L)}_n}\|^2_{L^2(\R_+\times {\mathbb  S}^{2N-1})},$$
which together with  \eqref{est**} and  \eqref{est***} implies that 
$$ \limsup_{n\to\infty}\;  \| \wt R^{(L)}_n\|_{\mathcal L} \leq C  {\mathcal \eta}({\rm r}^{(L)}_{n})^{\frac 1 2}\,.$$
The fact that $\ds  {\mathcal \eta}({\rm r}^{(L)}_{n})\stackrel{L\to\infty}\longrightarrow 0$  allows to conclude  the proof of  \eqref{heart}.

    \end{proof}
   
   \medbreak
   
\noindent  In order to obtain a decomposition under the form  \eqref{decompN},  we shall replace the profiles obtained above in $L^2 (\R_+\times {\mathbb  S}^{2N-1})$ by their average on the sphere. To go to this end, let us point out that if we set  
$$ \Phi_n^l(x):=   \frac {C_{N}} { \sqrt{ \alpha_n}}\int_{\R^{ 2 N }}\frac{ {\rm e}^{ i \, (x-x^l_n) \cdot \xi}}{|\xi|^{ 2 N }}   \;\Big[ {\Phi^{(l)}} \Big(\frac {\log |\xi| } {\alpha_n}, \omega \Big)- {\phi^l} \Big(\frac {\log |\xi| } {\alpha_n} \Big)\Big] \, d \xi\,,$$ 
where $\ds  {\phi^l}(t):= \frac {1} {\omega_{2N-1}} \int_{{\mathbb  S}^{2N-1}}  \; {\Phi^{(l)}} (t, \omega ) \, d  \omega$, 
then we have
\begin{equation}  \label{heart1} \;\|\Phi^l_n\|_{\mathcal
L}\stackrel{n\to\infty}\longrightarrow 0\,.\end{equation}
Assertion  \eqref{heart1} stems from  the following lemma:
 \begin{lem}
\label{avlem}
Let  $(\alpha_n)_{n \in \N}$ be a sequence
of positive real numbers going to infinity, $  {\Phi}$ be  a function in $L^2(\R_+\times {\mathbb  S}^{2N-1})$ and set for $x\in \R^{ 2 N }$ 
$$ \Phi_n(x):=  \frac {1} { (2\pi)^{ 2 N }} \sqrt{ \frac {1} { \alpha_n}}\int_{\R^{ 2 N }}\frac{ {\rm e}^{ i \,x \cdot \xi}}{|\xi|^{ 2 N }}   \;\Big[ {\Phi} \Big(\frac {\log |\xi| } {\alpha_n}, \omega \Big)- {\phi} \Big(\frac {\log |\xi| } {\alpha_n} \Big)\Big] \, d \xi\,,$$ 
where
\begin{equation}  \label{condav} {\phi}(t):= \frac {1} {\omega_{2N-1}} \int_{{\mathbb  S}^{2N-1}}  \; {\Phi} (t, \omega ) \, d  \omega.\end{equation}
Then
\begin{equation}  \label{heart2} \;\|\Phi_n\|_{\mathcal
L}\stackrel{n\to\infty}\longrightarrow 0\,.\end{equation}
 \end{lem}
\begin{proof}
 Let us first consider the case when  $  {\Phi} $ belongs to ${\cD}(\R_+\times {\mathbb  S}^{2N-1})$ and  decompose $ \Phi_n (x)$ as follows:
  $$\Phi_n (x)= \Phi^{(1)}_n (x)+ \Phi^{(2)}_n(x), \quad \mbox{with} \quad $$
  $$  \Phi^{(1)}_n (x):= \frac {1} { (2\pi)^{ 2 N }} \sqrt{ \frac {1} { \alpha_n}}\int_{{ 1\leq |\xi| \leq \max(\frac{1}{|x|},1)}}\frac{ {\rm e}^{ i \,x \cdot \xi}}{|\xi|^{ 2 N }}   \;\Big[ {\Phi} \Big(\frac {\log |\xi| } {\alpha_n}, \omega \Big)- {\phi} \Big(\frac {\log |\xi| } {\alpha_n} \Big)\Big] \, d \xi\,\cdot$$
  Observe that  in view of \eqref{condav} we have 
  $$ \sqrt{ \frac {1} { \alpha_n}}\int_{{ 1\leq |\xi| \leq \max(\frac{1}{|x|},1)}}\frac{1}{|\xi|^{ 2 N }}   \;\Big[ {\Phi} \Big(\frac {\log |\xi| } {\alpha_n}, \omega \Big)- {\phi} \Big(\frac {\log |\xi| } {\alpha_n} \Big)\Big] \, d \xi\,= 0.$$
Therefore
 \begin{eqnarray*}\big| \Phi^{(1)}_n (x) \big|&=& \frac {1} { (2\pi)^{ 2 N }} \sqrt{ \frac {1} { \alpha_n}} \, \Big|  \int_{{ 1\leq |\xi| \leq \max(\frac{1}{|x|},1)}}\frac{ ({\rm e}^{ i \,x \cdot \xi}-1)}{|\xi|^{ 2 N }}   \;\Big[ {\Phi} \Big(\frac {\log |\xi| } {\alpha_n}, \omega \Big)- {\phi} \Big(\frac {\log |\xi| } {\alpha_n} \Big)\Big] \, d \xi \Big|  \\ & \lesssim & \frac{|x|}{\sqrt{\alpha_n}} \int_1^{\max(\frac{1}{|x|},1)} \int_{{\mathbb  S}^{2N-1}}\Big| {\Phi} \Big(\frac {\log \rho} {\alpha_n}, \omega \Big)- {\phi} \Big(\frac {\log \rho } {\alpha_n} \Big)\Big| \, d\rho \, d \omega \\ & \lesssim & \frac{\;1}{\sqrt{\alpha_n}}\, \cdot \end{eqnarray*}
 Now  to estimate $\Phi^{(2)}_n$, we shall argue as for the proof of  Assertion \eqref{comp}.  Thus integrating by parts, we deduce that
   $ \Phi^{(2)}_n(x) = \Phi^{(2,1)}_n (x)+ \Phi^{(2,2)}_n(x)$, with
    \begin{eqnarray*} \Phi^{(2,1)}_n(x) &= & \frac {i} { (2\pi)^{ 2 N }} \sqrt{ \frac {1} { \alpha_n}}\int_{  |\xi| \geq  \max(\frac{1}{|x|},1)}  {\rm e}^{ i x \cdot \xi} \frac {x} {|x|^{2}}\cdot \nabla_\xi\Big[\frac{ 1}{|\xi|^{ 2 N }} \Big( {\Phi} \Big(\frac {\log |\xi| } {\alpha_n}, \omega \Big)- {\phi} \Big(\frac {\log |\xi| } {\alpha_n} \Big)\Big)\Big]   d \xi,  \\ \|\Phi^{(2,2)}_n\|_{L^\infty} & \lesssim & \frac{1}{\sqrt{\alpha_n}} \cdot \end{eqnarray*}  
   By straightforward  computations, one can prove that    \begin{eqnarray*}  \Big|\Phi^{(2,1)}_n (x)\Big| \lesssim   \frac{1}{\sqrt{\alpha_n} |x|} \int_{\rho \geq  \max(\frac{1}{|x|},1)}  \frac{d \rho}{\rho^2}  \lesssim   \frac{\;1}{\sqrt{\alpha_n}}\, \cdot\end{eqnarray*}  
    In summary, in the case when $  {\Phi} $ belongs to ${\cD}(\R_+\times {\mathbb  S}^{2N-1})$, we have proved  that
    $$  \|\Phi_n\|_{L^\infty}   \lesssim  \frac{1 }{\sqrt{\alpha_n}} \, \cdot$$ 
   This achieves the proof of the result, noticing that $\| \Phi_n\|_{L^2} \stackrel{n\to\infty}\longrightarrow  0$ and remembering  that for bounded functions the Orlicz space ${\cL}$ acts like $L^2$. \\
    
 \noindent Let us now treat the case when  $  {\Phi} \in L^2(\R_+\times {\mathbb  S}^{2N-1})$. Density arguments ensure that 
for  any $\varepsilon>0$,  there exists $  {\Phi}_\varepsilon\in{\cD}(\R_+\times {\mathbb  S}^{2N-1})$ such that 
$$
\|  {\Phi}-  {\Phi}_\varepsilon\|_{L^2}\leq \varepsilon\,.
$$
Therefore $\varepsilon>0$ being fixed, we can write 
$$ \Phi_n(x)=  \frac {1} { (2\pi)^{ 2 N }} \sqrt{ \frac {1} { \alpha_n}}\int_{\R^{ 2 N }}\frac{ {\rm e}^{ i \,x \cdot \xi}}{|\xi|^{ 2 N }}   \;\Big[ {\Phi}_\varepsilon \Big(\frac {\log |\xi| } {\alpha_n}, \omega \Big)- {\phi}_\varepsilon \Big(\frac {\log |\xi| } {\alpha_n} \Big)\Big] \, d \xi +{\rm t}_{n,\varepsilon}(x)\,,$$ 
where  $ \ds  {\phi}_\varepsilon(t):= \frac {1} {\omega_{2N-1}} \int_{{\mathbb  S}^{2N-1}}  \; {\Phi}_\varepsilon (t, \omega ) \, d  \omega $ and $ \| {\rm t}_{n,\varepsilon}\|_{ H^N(\R^{ 2 N })}  \lesssim \varepsilon$.   This ends the proof of the result by virtue of the Sobolev embedding \eqref{embedorliczNN}.

\end{proof}

\medbreak

 \begin{rem}
Let us also note that 
$$ \|\Phi_n\|^2_{\dot H^N(\R^{2N})}= \|\Phi\|^2_{L^2(\R_+\times {\mathbb  S}^{2N-1})} - \omega_{2N-1}\,\|\phi\|^2_{L^2(\R_+)}\,.$$ 
  \end{rem}
  
  \medbreak

\noindent Now let us return to the proof of Theorem \ref{main2}. In view of the above analysis, up to a subsequence extraction, we obtain a decomposition of $(u_n)$ under the form 
 \eqref{first}, and for any $j \geq 1$ we have  
\begin{equation}  \label{partj}g^{(j)}_n (x)=  \sum_{k=1}^K    \frac {C_{N}} { \sqrt{\alpha_n^{(j)}}}\int_{\R^{ 2 N }}\frac{ {\rm e}^{ i \, (x-x^{(j,k)}_n) \cdot \xi}}{|\xi|^{ 2 N }}   \; {\phi^{(j,k)}} \Big(\frac {\log |\xi| } {\alpha^{(j)}_n} \Big) \, d \xi+{\cR}^{(j,k)}_n(x)\, ,\end{equation}
 where the couples $(\underline{x}^{(j,k)},\phi^{(j,k)} )$ are pairwise
orthogonal in the sense of \eqref{caseII}, and where 
 \begin{equation}  \label{remj} \limsup_{n\to\infty}\;\|{\cR}^{(j,K)}_n\|_{\mathcal
L}\stackrel{K\to\infty}\longrightarrow 0, \quad \mbox{and}\end{equation}
\begin{equation}  \label{eqj} \|
g_n^{(j)}\|_{\dot H^N(\R^{2N})}^2= \Sum_{k=1}^{K}\,  \|  {\phi^{(j,k)} }\|_{L^2(\R_+)}^2+\|
{\cR}^{(j,k)}_n\|_{\dot H^N(\R^{2N})}^2+\circ(1),\quad n\to\infty\,.\end{equation}
Summing  Decomposition \eqref{partj}, we deduce that  up to a subsequence extraction
$$ u_n(x)= \Sum_{j=1}^{L} \sum_{k=1}^{K_j}  \frac {C_{N}} { \sqrt{ \alpha_n^{(j)}}}\int_{\R^{ 2 N }}\frac{ {\rm e}^{ i  (x-x^{(j,k)}_n) \cdot \xi}}{|\xi|^{ 2 N }}{\phi^{(j,k)}} \Big(\frac {\log |\xi| } {\alpha^{(j)}_n} \Big) d \xi+{\rm r}^{(L,k_1,\cdots,K_L)}_n(x),$$
with under notation of  \eqref{first} and \eqref{partj}
$${\rm r}^{(L,k_1,\cdots,K_L)}_n=  {\rm t}^{(L)}_n +  \Sum_{j=1}^{L} {\cR}^{(j,K_j)}_n\,.$$ 
Since $(|D|^N{\rm t}^{(L)}_n)$ is  $\log $-unrelated to any scale $(\alpha_n^{(j)})$,  $j =1,\cdots,L$, and $(|D|^N g_n^{(j)})$  is $\alpha^{(j)}$ $\log $-oscillating where the scales $\alpha^{(j)}$ are pairwise
orthogonal in the sense of \eqref{caseI}, we deduce from   \eqref{ortogonall2}  and \eqref{eqj} that for any $L$ and $K_j$
\begin{equation} \label{ort**} \|
u_n\|_{\dot H^N(\R^{2N})}^2= \Sum_{j=1}^{L} \Sum_{k=1}^{K_j}\,  \|  {\phi^{(j,k)} }\|_{L^2(\R_+)}^2+\|
{\rm r}^{(L,k_1,\cdots,K_L)}_n\|_{\dot H^N(\R^{2N})}^2+\circ(1),\quad n\to\infty\,,\end{equation} 
which concludes the proof of the theorem.

%%%%%%%%%%%%%%%%%%%%%%%%%%%%%%%%%%%%%%%%%%%%%

%%%%%%%%%%%%%%%%%%%%%%%%%%%%%%%%%%%%%%%%%%%

\section{Appendix}\label{Appendix}

%@@@@@@@@@@@@@@@@@@@@@@@@@@@@@@@@@@@@%@@@@@@@@@@@@@@@@@@@@@@@@@@@@@@@@@@@@%@@@@@@@@

\subsection{ Connection between ${\mathcal B}$ and Orlicz spaces} 
\label{useful*}
The following result allows to relate the Orlicz norm with the $ \|\cdot\|_{\mathcal B}$ norm:
 \begin{prop}
\label{uspprop*}
There is a positive constant $C$ such that 
$$ \|w\|_{  {\mathcal
L}}  \leq C  \|\wt {w}\|_{\mathcal B}, $$ 
where \begin{equation}
\label{eq*} w(x):= \frac {1} { (2\pi)^{2N}} \, \int_{|\xi| \geq1}\frac{ {\rm e}^{ i \, x \cdot \xi} \,}{|\xi|^{2N}}   \; {\wt {w}}(\log |\xi|,\omega) \, d \xi \, ,\end{equation}
with $ {\wt {w}}$  in $L^2 (\R_+\times {\mathbb  S}^{2N-1})$ and $$\|\wt {w}\|_{\mathcal B}:=\sup_{j\in\Z}\; \Big(\int_{{\mathbb  S}^{2N-1}} d  \omega \,\int_{2^j}^{2^{j+1}}\big| \; {\wt{w}}\big(t,\omega \big)\big|^2 \, d t \Big)^{\frac  12 }.$$  \end{prop}
\begin{proof}
For fixed $\lambda >0$, let us estimate the integral:
 $$
\int_{\R^{2N}}\;\Big({\rm
e}^{|\frac{w (x)}{\lambda}|^2}-1\Big)\,dx\,.
$$
Obviously
$$ \int_{\R^{2N}}\;\Big({\rm
e}^{|\frac{w(x)}{\lambda}|^2}-1\Big)\,dx= \sum_{p \geq 1} \frac{\|w\|^{2p }_{L^{2p }}}{\lambda^{2p } p !} \, \cdot$$ 
Firstly let us investigate $\|w\|^{2p }_{L^{2p }}$. We have that for any $p\geq 1$
$$\|w\|^{2p }_{L^{2p }} \leq C^{2p } \|\wh {w}\|^{2p }_{L^{\frac {2p }{2p-1 } }}.$$ 
But in view of \eqref{eq*} and H\"older inequality, we get for any $p \geq 2$
 \begin{eqnarray*}  \|\wh {w}\|^{\frac {2p }{2p-1 } }_{L^{\frac {2p }{2p-1 } }} &\lesssim& \|\wt {w}\|^{\frac {2p }{2p-1 } }_{\mathcal B} \sum_{j\in\Z} \Big(\int_{2^j}^{2^{j+1}} {\rm
e}^{- \frac {2Nt }{p-1 }} dt\Big)^{\frac {p-1 }{2p-1 } }\,.\end{eqnarray*}
\noindent 
Clearly, 
$$ \Big(\int_{2^j}^{2^{j+1}} {\rm
e}^{- \frac {2Nt }{p-1 }} dt\Big)^{\frac {p-1 }{2p-1 } } = \Big(\frac {p-1 }{2N}  \Big)^{\frac {p-1 }{2p-1 } } \Big( {\rm
e}^{ \frac {-2N2^j }{p-1 }} -  {\rm
e}^{ \frac {-2N2^{j+1} }{p-1 }} \Big)^{\frac {p-1 }{2p-1 } }\, .$$ 
Choosing $j_0$ such that $\ds \frac 1 2 \leq \frac {2^{j_0} }{2p-1 } \leq 1$, we infer that 
$$ \sum_{j \geq j_0} \Big(\int_{2^j}^{2^{j+1}} {\rm
e}^{- \frac {2Nt }{p-1 }} dt\Big)^{\frac {p-1 }{2p-1 } } \lesssim \big(2 p-1  \big)^{\frac {p-1 }{2p-1 } }\,.$$
Indeed observing that the ratio $\ds \frac {p-1 }{2p-1} $  is uniformly bounded with respect to $p \geq 2$,
we deduce that for all $j$ 
$$ \Big(\int_{2^j}^{2^{j+1}} {\rm
e}^{- \frac {2Nt }{p-1 }} dt\Big)^{\frac {p-1 }{2p-1 } }\lesssim   \big(2 p-1  \big)^{\frac {p-1 }{2p-1 } } {\rm
e}^{- \frac {2N2^{j} }{2p-1 }} \,,$$
which according to the choice of $j_0 $  gives rise to 
 \begin{eqnarray*} \sum_{j \geq j_0} \Big(\int_{2^j}^{2^{j+1}} {\rm
e}^{- \frac {2Nt }{p-1 }} dt\Big)^{\frac {p-1 }{2p-1 } } &\lesssim& \big(2 p-1  \big)^{\frac {p-1 }{2p-1 } }  \sum_{j \geq j_0}  {\rm
e}^{- \frac {2N2^{j} }{2p-1 }} \\ &\lesssim  &  \big(2 p-1  \big)^{\frac {p-1 }{2p-1 } }\sum_{j \geq j_0}  \big(2 p-1  \big) 2^{-j } \\ &\lesssim  &  \big(2 p-1  \big)^{\frac {p-1 }{2p-1 } }  \big(2 p-1  \big) 2^{-j_0 }\lesssim \big(2 p-1  \big)^{\frac {p-1 }{2p-1 } }\,.\end{eqnarray*} 
 On the other hand for $j \leq j_0$, we get 
 \begin{eqnarray*}  \Big(\int_{2^j}^{2^{j+1}} {\rm
e}^{- \frac {2Nt }{p-1 }} dt\Big)^{\frac {p-1 }{2p-1 } } &=& \Big(\frac {p-1 }{2N}  \Big)^{\frac {p-1 }{2p-1 } } {\rm
e}^{- \frac {2N2^{j+1} }{2p-1 }} \Big( {\rm
e}^{ \frac {2N2^j }{p-1 }} -  1\Big)^{\frac {p-1 }{2p-1 } }\\ &\lesssim  & \Big(2p-1 \Big)^{\frac {p-1 }{2p-1 } }\Big(\frac {2^j }{2p-1 }  \Big)^{\frac {p-1 }{2p-1 } }\lesssim   \big(2^j  \big)^{\frac {p-1 }{2p-1 } }\, .\end{eqnarray*} 
In view of the choice of $j_0$, this implies that $$ \sum_{ j \leq j_0} \Big(\int_{2^j}^{2^{j+1}} {\rm
e}^{- \frac {2Nt }{p-1 }} dt\Big)^{\frac {p-1 }{2p-1 } }\lesssim   \sum_{ j \leq j_0} \big(2^{\frac {p-1 }{2p-1 } }  \big)^j \lesssim 2^{( j_0+1)\frac {p-1 }{2p-1 } }\lesssim \big(2 p-1  \big)^{\frac {p-1 }{2p-1 } }\,.$$
By virtue of the above, this ensures  that 
$$ \sum_{j\in\Z} \Big(\int_{2^j}^{2^{j+1}} {\rm
e}^{- \frac {2Nt }{p-1 }} dt\Big)^{\frac {p-1 }{2p-1 } } \lesssim \big(2 p-1  \big)^{\frac {p-1 }{2p-1 } }\,.$$
We deduce that for any $p \geq 2$
$$\|w \|^{2p }_{L^{2p }} \lesssim C^{2p } \,  \|\wt {w}\|^{2p }_{\mathcal B} \,  \big(2 p-1  \big)^{p-1  }\,.$$ 
Along the same lines, we obtain 
$$ \|w\|^{2 }_{L^{2 }} \lesssim  \|\wt {w}\|^{2 }_{\mathcal B}\,,$$ 
which leads to 
$$ \int_{\R^{2N}}\;\Big({\rm
e}^{|\frac{w(x)}{\lambda}|^2}-1\Big)\,dx\lesssim  \sum_{p \geq 1} \frac{C^{2p }\,  \|\wt {w}\|^{2p }_{\mathcal B} \,  \big(2 p-1  \big)^{p-1  }}{\lambda^{2p } p !} \, \cdot$$ 
This ends the proof of the proposition thanks to  Stirling formula.\end{proof}

\medbreak

\subsection{ Fourier approximation of generalizations of example by Moser}
\label{genemoser}
The aim of this paragraph is to establish Proposition \ref{genemoserN} which  allows to relate the generalizations of example by Moser to the elementary concentrations involving in Decomposition \eqref{decompN}.
  \begin{proof}[Proof of Proposition \ref{genemoserN} ]
The proof of  Proposition \ref{genemoserN} goes along the proof of  Lemma~\ref{avlem}. We sketch it here for the convenience of the reader. Before entering into the details, let us point out that the function  $$ \wt g_n(x):= \wt C_N \,\sqrt{\alpha_n}\;\psi\left(\frac{-\log|x|}{\alpha_n}\right)$$ is supported in the unit ball of $\R^{ 2 N }$ and also writes:
\begin{equation}\label{eq0}   \wt g_n(x)=   \frac {C_{N}} {  \sqrt{\alpha_n}} \int_{1 \leq |\xi| \leq \frac{ 1}{|x|}}\frac{ 1}{|\xi|^{ 2 N }}   {\varphi} \Big(\frac { \log |\xi|} {\alpha_n}\Big) \, d \xi\,.\end{equation}
 \noindent 
Now by density arguments, we can  as in the proof of Lemma \ref{avlem} reduce to the case when $  \varphi $ belongs to ${\cD}(\R_+)$. Let us then decompose $ g_n (x)$ as follows:
  $$g_n (x)= g^{(1)}_n (x)+ g^{(2)}_n(x), \quad \mbox{with} \quad $$
  $$  g^{(1)}_n (x):=  \frac {C_{N}} { \sqrt{ \alpha_n}}\int_{{ 1\leq |\xi| \leq \max(\frac{1}{|x|},1)}}\frac{ {\rm e}^{ i \,x \cdot \xi}}{|\xi|^{ 2 N }}   \;{\varphi} \Big(\frac { \log |\xi|} {\alpha_n}\Big)\, d \xi\,\cdot$$ 
 \noindent In one hand  in view of \eqref{eq0}, we have 
$$  g^{(1)}_n (x)=  \wt g_n(x) +  {\rm t}^1_n(x), \quad \mbox{with}$$  
 \begin{eqnarray*}\big| {\rm t}^1_n(x) \big|&=& \frac {C_{N}} {  \sqrt{ \alpha_n}}\Big|  \int_{{ 1\leq |\xi| \leq \max(\frac{1}{|x|},1)}}\frac{ ({\rm e}^{ i \,x \cdot \xi}-1)}{|\xi|^{ 2 N }}    \;{\varphi} \Big(\frac { \log |\xi|} {\alpha_n}\Big)\, \, d \xi \Big|  \\ & \lesssim & \frac{|x|}{\sqrt{\alpha_n}} \int_1^{\max(\frac{1}{|x|},1)} \Big| {\varphi} \Big(\frac {\log \rho} {\alpha_n} \Big)\Big| \, d\rho \lesssim \frac{\;1}{\sqrt{\alpha_n}}\, \cdot  \end{eqnarray*}
On the other hand integrating by parts, we get  $g^{(2)}_n(x)=  g^{(2,1)}_n(x)+g^{(2,2)}_n(x)$ with 
\begin{eqnarray*} g^{(2,1)}_n(x)&=&   \frac {i  \, C_{N}} { \sqrt{ \alpha_n}}\int_{  |\xi| \geq  \max(\frac{1}{|x|},1)}  {\rm e}^{ i x \cdot \xi} \frac {x} {|x|^{2}}\cdot \nabla_\xi\Big[\frac{ 1}{|\xi|^{ 2 N }} \;{\varphi} \Big(\frac { \log |\xi|} {\alpha_n}\Big)\,\Big]   d \xi \quad \mbox{and} \\ \|g^{(2,2)}_n\|_{L^\infty} & \lesssim & \frac{\;1}{\sqrt{\alpha_n}}\, \cdot\end{eqnarray*}
 Clearly  \begin{eqnarray*}  \Big|g^{(2,1)}_n (x)\Big| \lesssim   \frac{1}{\sqrt{\alpha_n} |x|} \int_{\rho \geq  \max(\frac{1}{|x|},1)}  \frac{d \rho}{\rho^2}  \lesssim   \frac{\;1}{\sqrt{\alpha_n}}\, \cdot\end{eqnarray*}  
  This easily ensures the result. \end{proof}

%@@@@@@@@@@@@@@@@@@@@@@@@@@@@@@@@@@@@%@@@@@@@@@@@@@@@@@@@@@@@@@@@@@@@@@@@@%@@@@@@@@

%@@@@@@@@@@@@@@@@@@@@@@@@@@@@@@@@@@@@%@@@@@@@@@@@@@@@@@@@@@@@@@@@@@@@@@@@@%@@@@@@@@

%@@@@@@@@@@@@@@@@@@@@@@@@@@@@@@@@@@@@%@@@@@@@@@@@@@@@@@@@@@@@@@@@@@@@@@@@@%@@@@@@@@

\end{document}